\documentclass{amsart}

\usepackage{amssymb}
\usepackage{graphicx} 
\usepackage{mathrsfs}
\usepackage{enumerate}
\usepackage{xspace}
\usepackage{color}
\usepackage{enumerate}
\usepackage{enumitem}
\usepackage{amsthm}
\usepackage{dsfont}
\usepackage{bbm}
\usepackage[colorlinks=true, linkcolor = blue, citecolor = blue]{hyperref}
\usepackage[numbers]{natbib}
\usepackage[backgroundcolor=white,bordercolor=red]{todonotes}
\DeclareMathAlphabet{\mathpzc}{OT1}{pzc}{m}{it}
\usepackage[nameinlink]{cleveref}
\numberwithin{equation}{section}
\begin{document}

\theoremstyle{plain}

\newtheorem{theorem}{Theorem}[section]
\newtheorem{lemma}[theorem]{Lemma}
\newtheorem{example}[theorem]{Example}
\newtheorem{proposition}[theorem]{Proposition}
\newtheorem{corollary}[theorem]{Corollary}
\newtheorem{definition}[theorem]{Definition}
\newtheorem{Ass}[theorem]{Assumption}
\newtheorem{condition}[theorem]{Condition}
\theoremstyle{definition}
\newtheorem{remark}[theorem]{Remark}
\newtheorem{SA}[theorem]{Standing Assumption}

\newcommand{\of}{[\hspace{-0.06cm}[}
\newcommand{\gs}{]\hspace{-0.06cm}]}

\newcommand\llambda{{\mathchoice
		{\lambda\mkern-4.5mu{\raisebox{.4ex}{\scriptsize$\backslash$}}}
		{\lambda\mkern-4.83mu{\raisebox{.4ex}{\scriptsize$\backslash$}}}
		{\lambda\mkern-4.5mu{\raisebox{.2ex}{\footnotesize$\scriptscriptstyle\backslash$}}}
		{\lambda\mkern-5.0mu{\raisebox{.2ex}{\tiny$\scriptscriptstyle\backslash$}}}}}

\newcommand{\1}{\mathds{1}}

\newcommand{\F}{\mathbf{F}}
\newcommand{\G}{\mathbf{G}}

\newcommand{\B}{\mathbf{B}}

\newcommand{\M}{\mathcal{M}}

\newcommand{\la}{\langle}
\newcommand{\ra}{\rangle}

\newcommand{\lle}{\langle\hspace{-0.085cm}\langle}
\newcommand{\rre}{\rangle\hspace{-0.085cm}\rangle}
\newcommand{\blle}{\Big\langle\hspace{-0.155cm}\Big\langle}
\newcommand{\brre}{\Big\rangle\hspace{-0.155cm}\Big\rangle}

\newcommand{\X}{\mathsf{X}}

\newcommand{\tr}{\operatorname{tr}}
\newcommand{\N}{{\mathbb{N}}}
\newcommand{\cadlag}{c\`adl\`ag }
\newcommand{\on}{\operatorname}
\newcommand{\oP}{\overline{P}}
\newcommand{\oO}{\mathcal{O}}
\newcommand{\D}{D(\mathbb{R}_+; \mathbb{R})}
\newcommand{\dd}{\mathsf{d}}

\renewcommand{\epsilon}{\varepsilon}

\newcommand{\fPs}{\mathfrak{P}_{\textup{sem}}}
\newcommand{\fPas}{\mathfrak{P}^{\textup{ac}}_{\textup{sem}}}
\newcommand{\rrarrow}{\twoheadrightarrow}
\newcommand{\cA}{\mathcal{C}}
\newcommand{\cR}{\mathcal{R}}
\newcommand{\cK}{\mathcal{K}}
\newcommand{\cQ}{\mathcal{Q}}
\newcommand{\cF}{\mathcal{F}}
\newcommand{\cC}{\mathcal{C}}
\newcommand{\cD}{\mathcal{D}}
\newcommand{\bC}{\mathbb{C}}
\newcommand{\bth}{\overset{\leftarrow}\theta}
\renewcommand{\th}{\theta}

\newcommand{\bR}{\mathbb{R}}
\newcommand{\nnabla}{\nabla}
\newcommand{\f}{\mathfrak{f}}
\newcommand{\g}{\mathfrak{g}}
\newcommand{\oconv}{\overline{\on{co}}\hspace{0.075cm}}
\renewcommand{\a}{\mathfrak{a}}
\renewcommand{\b}{\mathfrak{b}}
\renewcommand{\d}{d}
\newcommand{\bS}{\mathbb{S}^\d_+}
\newcommand{\p}{\dot{\partial}}
\newcommand{\dr}{r} 
\newcommand{\m}{\mathbb{M}}
\newcommand{\Q}{Q}
\newcommand{\ob}{\widehat{b}}
\newcommand{\osigma}{\widehat{\sigma}}

\renewcommand{\emptyset}{\varnothing}

\allowdisplaybreaks

\makeatletter
\@namedef{subjclassname@2020}{%
	\textup{2020} Mathematics Subject Classification}
\makeatother

 \title[Convergence of Viscosity Solutions to HJB PPDEs]{A convergence theorem for Crandall--Lions viscosity solutions to path-dependent Hamilton--Jacobi--Bellman PDEs}
\author[D. Criens]{David Criens}
\address{Albert-Ludwigs University of Freiburg, Ernst-Zermelo-Str. 1, 79104 Freiburg, Germany}
\email{david.criens@stochastik.uni-freiburg.de}



\thanks{
}
\date{\today}

\maketitle

\begin{abstract}
	We establish a convergence theorem for Crandall--Lions viscosity solutions to path-dependent Hamilton--Jacobi--Bellman PDEs. Our proof is based on a novel convergence theorem for dynamic sublinear expectations and the stochastic representation of viscosity solutions as value functions.
\end{abstract}
\section{Introduction}
Path-dependent partial differential equations (PPDEs) (that are PDEs in the space of continuous functions) are currently a very active field of research. 
Typically, they do not have smooth solutions, which motivates the study of weaker solution concepts, such as {\em viscosity solutions}, which are frequently used in classical finite dimensional frameworks. 

In the context of viscosity solutions to PPDEs, many different concepts have been introduced in the past years. We refer to the introductions of \cite{cosso,cosso2,zhou} for detailed comments on the literature. 
One of the most recent streams investigates viscosity solutions in the classical sense of Crandall and Lions (see, for instance, their seminal papers \cite{CL1,CL2}). We are only aware of the articles \cite{cosso,cosso2,zhou} on this framework, where, under different assumptions, existence and uniqueness results and stochastic representations are established. Further, the paper \cite{zhou} provides a stability result under uniform convergence assumptions.

In this paper, we continue this line of research in form of a limit theorem that shows that compact convergence of coefficients and initial functions propagates to the corresponding viscosity solutions. To provide a more precise statement, let \(v^n \colon [0, T] \times C([0, T]; \bR^d)\to  \bR\) be the (in some sense) unique bounded Crandall--Lions viscosity solution to the nonlinear PPDE
\begin{equation} \label{eq: PPDE intro}
	\begin{cases}   
		\p v^n (t, \omega) + G^n (t, \omega, v^n) = 0, & \text{for } (t, \omega) \in [0, T) \times C([0, T]; \bR^d), \\
		v^n (T, \omega) = \psi^n (\omega), & \text{for } \omega \in C([0, T]; \bR^d),
	\end{cases}
\end{equation}
where \(\psi^n \colon C ([0, T]; \bR^d) \to \mathbb{R}\) is bounded and continuous,
\begin{align*}
	G^n(t, \omega, \phi) := \sup \Big\{ \langle \nabla \phi (t&, \omega), b^n (f, t, \omega) \rangle
	\\&+ \tfrac{1}{2} \on{tr} \big[ \nabla^2\phi (t, \omega) \sigma^n (\sigma^n)^* (f, t, \omega) \big]\colon f \in F \Big\},
\end{align*}
and \(b^n \colon F \times [0, T] \times C ([0, T]; \bR^d)\to \mathbb{R}^d\) and \(\sigma^n \colon F  \times [0, T] \times C ([0, T]; \bR^d) \to \mathbb{R}^{d \times r}\) are sufficiently regular coefficients. Here, the action space \(F\) is compact and metrizable.
In our main Theorem~\ref{thm: stability}, we show that compact convergence of the input data, i.e., \(b^n \to b^0, \sigma^n \to \sigma^0, \psi^n \to \psi^0\) uniformly on compacts, implies compact convergence of the corresponding viscosity solutions, i.e., \(v^n \to v^0\) uniformly on compacts.

Our limit theorem can be compared to results in classical (that is time and point-dependent) frameworks, whose proofs typically rely either on stability, uniqueness and the Arzel\`a--Ascoli theorem or stability and a comparison result for discontinuous viscosity solutions, see \cite{barles} or \cite[Chapter~VII]{FS}.  It is worth mentioning that neither suitable stability results nor comparison results for discontinuous solutions seem to be available for our path-dependent framework. 

We propose a novel approach to show convergence, which relies on the stochastic representation of the solutions as dynamic sublinear expectations that was established in \cite{CN23a}. To explain the idea, let us recall this representation.
Define
\begin{align*}
	\cC^n(t,\omega) := \Big\{ P \in \mathfrak{P}_{\text{sem}}^{\text{ac}}(t)\colon &P(X = \omega \text{ on } [0, t]) = 1, 
	(\llambda \otimes P)\text{-a.e. } \\&\qquad (dB^P_{\cdot + t} /d\llambda, dC^P_{\cdot + t}/d\llambda) \in \Theta^n (\cdot + t, X) \Big\},
\end{align*}
where \(\fPas(t)\) denotes the set of semimartingale laws \(P\) after time \(t\) with absolutely continuous characteristics \((B^P_{\cdot + t}, C^P_{\cdot + t})\), \(X\) denotes the coordinate process, \(\llambda\) denotes the Lebesgue measure and
\[
\Theta^n (t, \omega) := \big\{(b^n ( f, t, \omega), \sigma^n (\sigma^n)^* (f, t, \omega)) \colon f \in F \big\}.
\]
It was proved in \cite{CN23a} that
\[
v^n (t, \omega) = \sup_{P \in \cA^n (t, \omega)} E^P \big[ \psi^n \big].
\]
By defining \(K := \{1/n \colon n \in \mathbb{N}\} \cup \{0\}\) and
\[
v (1/n, t, \omega) := \sup_{P \in \cA^n (t, \omega)} E^P \big[ \psi^n \big], \qquad v (0, t, \omega) := \sup_{P \in \cA^0 (t, \omega)} E^P \big[ \psi^0 \big],
\]
the sequence \((v^n)_{n = 0}^\infty\) translates to a function \(v \colon K \times [0, T] \times C([0, T]; \bR^d)\to \bR\) and the convergence \(v^n \to v^0\) on compacts follows from joint continuity of
\(
(k, t, \omega) \mapsto v(k, t, \omega).
\)
In other words, our main theorem can be deduced from a suitable continuity result for dynamic sublinear expectations. 

We establish such a result in a more abstract framework. Consider a first countable space \(K\), two coefficients \(b \colon K \times F \times \bR_+ \times C(\bR_+;\bR^\d) \to \bR^\d\) and \(\sigma \colon K \times F \times \bR_+ \times C(\bR_+;\bR^\d) \to \mathbb{R}^{d \times r}\) and set
\begin{align*}
	\cC(k, t,\omega) := \Big\{ P \in \mathfrak{P}_{\text{sem}}^{\text{ac}}(t)\colon &P(X = \omega \text{ on } [0, t]) = 1, 
	(\llambda \otimes P)\text{-a.e. } \\&\qquad (dB^P_{\cdot + t} /d\llambda, dC^P_{\cdot + t}/d\llambda) \in \Theta (k, \cdot + t, X) \Big\},
\end{align*}
with
\[
\Theta (k, t, \omega) := \big\{(b (k, f, t, \omega), \sigma \sigma^* (k, f, t, \omega)) \colon f \in F \big\}. 
\]
Under suitable assumptions on \(b\) and \(\sigma\), we prove joint continuity of the {\em extended value function} 
\[
(k, t, \omega) \mapsto v (k, t, \omega) := \sup_{P \in \cA (k, t, \omega)} E^P \big[ \psi (k, t, X) \big],
\]
where \(\psi \colon K \times\bR_+ \times C(\bR_+; \bR^d)\to \bR\) is bounded and continuous. 
As explained above, this continuity property implies our limit theorem for viscosity solutions.

The proof of this regularity result relies on Berge's maximum theorem. A similar approach was used in \cite{CN23a} for the \(k\)-independent case. In fact, upper hemicontinuity follows as in \cite{CN23a}. In this article, we focus on the proof for lower hemicontinuity that requires more sophisticated arguments to account for the \(k\)-dependence of the extended value function. At this stage, we adapt ideas from the paper \cite{jacod1981weak} on convergence of stochastic differential equations driven by general semimartingales.

The continuity of the extended value function is of interest in its own. It provides a convergence result for dynamic sublinear expectations such as random \(G\)-expectations as introduced in \cite{nutzRGE,NVH}. Latter are extensions of Peng's \(G\)-expectation (\cite{peng2007g,peng2008multi,peng2010}), which is an important model for Knightian uncertainty in financial mathematics.

The paper is structured as follows. In Section \ref{sec: vis} we formulate our main convergence theorem. Its proof is given in Section \ref{sec: EVF and pf stability}. The stochastic setting and the continuity of the extended value function can be found in Section~\ref{sec: setting}.

\section{A convergence result for nonlinear path-dependent PDEs} \label{sec: vis}
In this section we present the main result of this paper, namely a convergence theorem for Hamilton--Jacobi--Bellman PPDEs. 

Let \(\d, r\in \mathbb{N}\) be fixed dimensions and let \(T > 0\) be a finite time horizon. We define \(C ([0, T]; \bR^d)\) to be the space of all continuous functions from \([0, T]\) into \(\mathbb{R}^d\) endowed with the uniform topology.
Let \(F\) be a topological space and, for \(n \in \mathbb{Z}_+\), let \(b^n \colon F \times [0, T] \times C([0, T]; \bR^d) \to \mathbb{R}^d\) and \(\sigma^n \colon F \times [0, T] \times C([0, T]; \bR^d) \to \mathbb{R}^{d \times r}\) be Borel functions such that, for every \(f \in F\), \(b^n (f, t, \omega)\) depends on \((t, \omega)\) only through \((\omega (s))_{s < t}\). 
The set \[C^{1, 2}_{pol} ([0, T] \times C([0, T]; \bR^d); \mathbb{R}) =: \mathbb{C}^{1, 2}_{pol} ([0, T])\]
of test functions is defined in Appendix \ref{app: test}.
For every \(n \in \mathbb{N}\) and \(( t, \omega,\phi) \in [0, T] \times C([0, T]; \bR^d) \times\hspace{0.05cm} \mathbb{C}^{1, 2}_{pol} ([0, T]) \), we introduce the nonlinear operator
\begin{equation} \label{eq: G^n}
	\begin{split}
		G^n(t, \omega, \phi) := \sup \Big\{ \langle \nabla& \phi (t, \omega), b^n (f, t, \omega) \rangle
		\\&+ \tfrac{1}{2} \on{tr} \big[ \nabla^2\phi (t, \omega) \sigma^n (\sigma^n)^* (f, t, \omega) \big]\colon f \in F \Big\},
	\end{split}
\end{equation}
and the corresponding nonlinear PPDE
\begin{equation} \label{eq: PIDE}
	\begin{cases}   
		\p v^n (t, \omega) + G^n (t, \omega, v^n) = 0, & \text{for } (t, \omega) \in [0, T) \times C([0, T]; \bR^d), \\
		v^n (T, \omega) = \psi^n (\omega), & \text{for } \omega \in C([0, T]; \bR^d),
	\end{cases}
\end{equation}
where \(\psi^n \colon C([0, T]; \bR^d) \to \mathbb{R}\) is a bounded continuous function.

\begin{definition}[Viscosity Solution]
	A function \(u^n \colon [0, T] \times C([0, T]; \bR^d) \to \mathbb{R}\) is said to be a \emph{viscosity subsolution} to \eqref{eq: PIDE} if it is upper semicontinuous and the following two properties hold:
	\begin{enumerate}
		\item[\textup{(a)}] \(u^n(T, \cdot) \leq \psi^n\);
		\item[\textup{(b)}]
		for any \((t,\omega) \in [0, T) \times C([0, T]; \bR^d) \) and \( \phi \in \mathbb{C}^{1, 2}_{pol} ([0, T])\) satisfying
		\begin{align*}
			0 &= (u^n - \phi)(t,\omega) 
			\\&= \sup \big\{ (u^n - \phi)(s,\omega') \colon (s,\omega') \in [t, T] \times C([0, T]; \bR^d)  \big\},
		\end{align*}
		we have
		\(
		\p \phi (t, \omega) + G^n (t, \omega, \phi) \geq 0.
		\)
	\end{enumerate}
	Moreover, a function \(u^n \colon [0, T] \times C([0, T]; \bR^d) \to \mathbb{R}\) is said to be a \emph{viscosity supersolution} to \eqref{eq: PIDE} if it is lower semicontinuous and the following two properties hold:
	\begin{enumerate}
		\item[\textup{(a)}] \(u^n(T, \cdot) \geq \psi^n\);
		\item[\textup{(b)}]
		for any \((t,\omega) \in [0, T) \times C([0, T]; \bR^d)\) and \( \phi \in \mathbb{C}^{1, 2}_{pol} ([0, T])\) satisfying
		\begin{align*}
			0 &= (u^n - \phi)(t,\omega) 
			\\&= \inf \big\{ (u^n - \phi)(s,\omega') \colon (s,\omega') \in [t, T] \times C([0, T]; \bR^d) \big\},
		\end{align*}
		we have
		\(
		\p \phi (t, \omega) + G^n (t, \omega, \phi) \leq 0.
		\)
	\end{enumerate}
	Further, \(u^n\) is called a \emph{viscosity solution} if it is a viscosity sub- and supersolution. 
\end{definition}

Under suitable conditions, it was proved in \cite{cosso2,zhou} that the PPDE \eqref{eq: PIDE} has a unique bounded viscosity solution among all solutions with a certain (uniform or Lipschitz) continuity property. In \cite{CN23a} a stochastic representation as value function related to so-called nonlinear continuous semimartingales was established. The following condition collects the assumptions needed to use the uniqueness theorem from \cite{zhou} and the stochastic representation from~\cite{CN23a}.

\begin{condition} \label{cond: main2}
	For every \(n \in \mathbb{Z}_+\), the following properties hold:
	\begin{enumerate}
		\item[\textup{(i)}] \(F\) is a compact metrizable space.
		\item[\textup{(ii)}] \(b^n\) and \(\sigma^n\) are continuous on \(F \times [0, T] \times C([0, T]; \bR^d)\).
		\item[\textup{(iii)}] \(\Theta^n\) is convex-valued.
		\item[\textup{(iv)}] There exists a constant \(C = C_n > 0\) such that  
		\[
		\| b^n (f, t, \omega) \| + \|\sigma^n (f, t, \omega)\| \leq C \Big(1 + \sup_{s \in [0, t]} \|\omega (s)\| \Big)
		\]
		for all \((f, t, \omega) \in F \times [0, T] \times C([0, T]; \bR^d)\). 
		\item[\textup{(v)}] There exists a constant \(C = C_n > 0\) such that 
		\begin{align*}
			\| b^n (f, t, \omega) - b^n (f, t, \alpha) \| & \leq C \sup_{s \in [0, t]} \|\omega (s) - \alpha (s)\|,\\
			\|\sigma^n (f, t, \omega) - \sigma^n (f, t, \alpha)\| &\leq C \sup_{s \in [0, t]} \|\omega (s) - \alpha (s)\|
		\end{align*}
		for all \(\omega, \alpha \in C([0, T]; \bR^d)\) and \((f, t) \in F \times [0, T]\).
		\item[\textup{(vi)}] The function \(\psi^n\) is Lipschitz continuous, i.e., there exists a constant \(C = C_n > 0\) such that 
		\[
		\|\psi^n (\omega) - \psi^n (\alpha)\| \leq C \sup_{s \in [0, T]} \|\omega (s) - \alpha (s)\|
		\]
		for all \(\omega, \alpha \in C([0, T]; \bR^d)\). 
	\end{enumerate}
\end{condition}
We emphasize that the Lipschitz constants in (v) and (vi) are allowed to depend on \(n\). 
We define a function \(\dd \colon [0, T] \times C([0, T]; \bR^d) \times [0, T] \times C([0, T]; \bR^d) \to \bR_+\) by
\begin{equation} \label{eq: def d}
	\begin{split}
		\dd ( (t,\omega), (s, \alpha)) :=  \Big( 1 + \sup_{r \in [0, t]} \|\omega (r)\| + \ & \sup_{r \in [0, s]}  \|\alpha (r)\|\Big) | t - s |^{1/2} \\&+ \sup_{r \in [0, T]} \|\omega(r \wedge t) - \alpha( r \wedge s) \|.
	\end{split}
\end{equation}
A function \(f \colon [0, T] \times C([0, T]; \bR^d) \to \mathbb{R}\) is called {\em \(\dd\)-Lipschitz continuous}, if there exists a constant \(L > 0\) such that 
\[
| f (t, \omega) - f (s, \alpha) | \leq L \dd ((t, \omega), (s, \alpha))
\] 
for all \((t, \omega), (s, \alpha) \in [0, T] \times C([0, T]; \bR^d)\).

The following existence and uniqueness result is a consequence of  \cite[Theorem~6.2]{zhou}.

\begin{theorem} \label{thm: uni}
	Assume that Condition \ref{cond: main2} holds. For every \(n \in \mathbb{Z}_+\), the PPDE \eqref{eq: PIDE} has a bounded viscosity solution that is unique among all bounded \(\dd\)-Lipschitz continuous viscosity solutions.
\end{theorem}

Next, we provide a convergence result for the nonlinear PPDEs \eqref{eq: PIDE}.

\begin{condition} \label{cond: main3}
	\quad 
	\begin{enumerate}
		\item[\textup{(i)}]
		For all compact sets \(G \subset [0, T] \times C([0, T]; \bR^d)\) and \(A \subset C([0, T]; \bR^d)\), 
		\begin{align*}
			\sup \big\{ \|b^n (f, t, \omega) - b^0 (f, t, \omega)\| \colon (f, t, \omega) \in F \times G \big\} &\to 0, \\
			\sup \big\{ \|\sigma^n (f, t, \omega) - \sigma^0 (f, t, \omega) \| \colon (f, t, \omega)\in F \times G \big\} &\to 0, \\
			\sup \big\{ | \psi^n (\omega) - \psi^0 (\omega) |\colon \omega \in A \big\} &\to 0.
		\end{align*}
		\item[\textup{(ii)}]
		There exists a constant \(C > 0\) such that 
		\[
		\| b^n (f, t, \omega) \| + \|\sigma^n (f, t, \omega)\| \leq C \Big(1 + \sup_{s \in [0, t]} \|\omega (s)\| \Big)
		\]
		for all \(n \in \mathbb{Z}_+\) and \((f, t, \omega) \in F \times [0, T] \times C([0, T]; \bR^d)\). 
		\item[\textup{(iii)}] There exists a constant \(C > 0\) such that \(|\psi^n (\omega)| \leq C\) for all \(n \in \mathbb{Z}_+\) and \(\omega \in C([0, T]; \bR^d)\).
	\end{enumerate}
\end{condition}
The following theorem is our main result. Its proof is given in Section~\ref{sec: pf stability} below. 
\begin{theorem} \label{thm: stability}
	Assume that the Conditions \ref{cond: main2} and \ref{cond: main3} hold. For every \(n \in \mathbb{Z}_+\), let \(v^n\) be the unique, among all bounded \(\dd\)-Lipschitz continuous, viscosity solution to \eqref{eq: PIDE}. 
	Then, for every compact set \(G \subset [0, T] \times C([0, T]; \bR^d)\), 
	\[
	\sup_{ (t, \omega) \in G} | v^n (t, \omega) - v^0 (t, \omega) | \to 0 \text{ as } n \to \infty.
	\]
\end{theorem}

We end this section with an example that is related to the random \(G\)-expectations from \cite{nutzRGE,NVH}.

\begin{example} 
	\label{ex: random G BM}
	In this example we consider the case \(d = 1\).
	For \(n \in \mathbb{Z}_+\), let \(\underline{b}\hspace{0.5cm} \hspace{-0.5cm}^n, \overline{b} \hspace{0.5cm} \hspace{-0.5cm}^n \colon [0, T] \times C([0, T]; \bR) \to \mathbb{R}\) and \(\underline{a}^n, \overline{a}^n \colon [0, T] \times C([0, T]; \bR) \to (0, \infty)\) be continuous functions such that \(\underline{b}^n \leq \overline{b} \hspace{0.5cm} \hspace{-0.5cm} ^n\) and \(\underline{a}^n \leq \overline{a}^n\).
	To ease our presentation, we assume that there exists a constant \(C > 0\) such that 
	\[
	| \underline{b}\hspace{0.5cm} \hspace{-0.5cm}^n (t, \omega) |, | \overline{b}\hspace{0.5cm} \hspace{-0.5cm}^n (t, \omega) |  \leq C, \qquad \frac{1}{C} \leq \underline{a}^n (t, \omega), \overline{a}^n (t, \omega) \leq C
	\]
	for all \(n \in \mathbb{Z}_+\) and \((t, \omega) \in [0, T] \times C([0, T]; \bR)\).
	Further, for every \(n \in \mathbb{Z}_+\), suppose that there exists a constant \(C = C_n > 0\) such that 
	\begin{align*}
		| \underline{b}\hspace{0.5cm} \hspace{-0.5cm}^n (t, \omega) - \underline{b}\hspace{0.5cm} \hspace{-0.5cm}^n (t, \alpha) | + |\overline{b}\hspace{0.5cm} \hspace{-0.5cm}^n (t, \omega) - \overline{b}\hspace{0.5cm} \hspace{-0.5cm}^n (t, \alpha) | &\leq C \sup_{s \in [0, t]} | \omega (s) - \alpha (s) |, \\
		| \underline{a}\hspace{0.5cm} \hspace{-0.5cm}^n (t, \omega) - \underline{a}\hspace{0.5cm} \hspace{-0.5cm}^n (t, \alpha) | + |\overline{a}\hspace{0.5cm} \hspace{-0.5cm}^n (t, \omega) - \overline{a}\hspace{0.5cm} \hspace{-0.5cm}^n (t, \alpha) | &\leq C \sup_{s \in [0, t]} | \omega (s) - \alpha (s) |
	\end{align*}
	for all \(t \in [0, T]\) and \(\omega, \alpha \in C([0,T]; \bR)\). Let \(\psi^0, \psi^1, \dots\) be bounded Lipschitz continuous functions from \(C([0, T]; \bR)\) into \(\mathbb{R}\) such that Condition \ref{cond: main3} (iii) holds.
	
	We consider the nonlinear operator 
	\begin{align*}
		G^n (t, \omega, \phi) := \sup \big\{ \partial \phi  (t&, \omega) x \colon \underline{b}\hspace{0.5cm} \hspace{-0.5cm}^n (t, \omega) \leq x \leq \overline{b}\hspace{0.5cm} \hspace{-0.5cm}^n (t, \omega) \big\} \\&+ \tfrac{1}{2}\sup \big\{ \partial^2 \phi  (t, \omega) y \colon \underline{a}^n (t, \omega) \leq y \leq \overline{a}^n (t, \omega) \big\}
	\end{align*}
	for \((t, \omega, \phi) \in [0, T] \times C([0, T]; \bR) \times \mathbb{C}^{1, 2}_{pol} ([0, T])\). Let \(v^n\) be the unique, among all bounded \(\dd\)-Lipschitz continuous, viscosity solutions to~\eqref{eq: PIDE}.
	Theorem \ref{thm: stability} shows the following implication:
	\begin{align*}
		\underline{b}\hspace{0.5cm} \hspace{-0.5cm}^n \to \underline{b}\hspace{0.5cm} \hspace{-0.5cm}^0, \overline{b}\hspace{0.5cm} \hspace{-0.5cm}^n \to \overline{b}\hspace{0.5cm} \hspace{-0.5cm}^0, \underline{a}^n \to \underline{a}^0, \overline{a}^n \to \overline{a}^0, \psi^n \to \psi^0& \text{ compactly}\hspace{0.5cm} \\&\Longrightarrow \hspace{0.5cm} v^n \to v^0 \text{ compactly},
	\end{align*}
	where compact convergence refers to uniform convergence on compact subsets.
	
	More precisely, we can apply Theorem \ref{thm: stability} with \(F := [0, 1] \times [0, 1]\) and 
	\begin{align*}
		b^n (f, t, \omega) & := \underline{b}\hspace{0.5cm} \hspace{-0.5cm}^n (t, \omega) + f_1 \cdot (\overline{b}\hspace{0.5cm} \hspace{-0.5cm}^n (t, \omega) - \underline{b}\hspace{0.5cm} \hspace{-0.5cm}^n (t, \omega)),\\
		\sigma^n (f, t, \omega) &:= \sqrt{\underline{a}^n (t, \omega) + f_2 \cdot (\overline{a}^n (t, \omega) - \underline{a}^n (t, \omega))}, 
	\end{align*}
	for \(f = (f_1, f_2) \in F\) and \((t, \omega) \in [0, T] \times C([0, T]; \bR)\). 
\end{example}

\section{Extended Value Functions and the Proof of Theorem \ref{thm: stability}} \label{sec: EVF and pf stability}

In this section we prove our main Theorem \ref{thm: stability}. At the heart of its proof lies the stochastic representation of the functions \(v^0, v^1, \dots\) as value functions corresponding to so-called {\em nonlinear continuous semimartingales} as introduced in \cite{CN23a}. In Section~\ref{sec: setting}, we discuss an extension of the stochastic framework for this representation. The object of interest is a dynamic sublinear expectation that is called {\em extended value function}.
The key step of the proof for Theorem~\ref{thm: stability} is to establish joint continuity of the extended value function. The details for this are given in Section~\ref{sec: pf stability}, while the continuity of the extended value function is established in Section \ref{sec: pf regularity}.

\subsection{Extended value functions and their continuity}\label{sec: setting}
Without loss of generality, we work with an infinite time horizon. 
Define $\Omega$ to be the space of all continuous functions \(\mathbb{R}_+ \to \mathbb{R}^\d\) endowed with the local uniform topology. 
Further, we write \(X\) for the canonical process on the path space \(\Omega\), i.e., \(X_t (\omega) = \omega (t)\) for \(\omega \in \Omega\) and \(t \in \mathbb{R}_+\). 
It is well-known that \(\mathcal{F} := \mathcal{B}(\Omega) = \sigma (X_t, t \geq 0)\).
We define $(\mathcal{F}_t)_{t \geq 0}$ as the canonical filtration generated by $X$, i.e., \(\mathcal{F}_t := \sigma (X_s, s \leq t)\) for \(t \in \mathbb{R}_+\). 
Further, we denote the predictable \(\sigma\)-field by \(\mathscr{P}\).
The set of probability measures on \((\Omega, \mathcal{F})\) is denoted by \(\mathfrak{P}(\Omega)\) and endowed with the usual topology of convergence in distribution.
To lighten our notation, we use stochastic intervals, i.e., for two stopping times \(S\) and \(T\), we set
\[
\of S, T \of \hspace{0.1cm} := \{ (t,\omega) \in \bR_+ \times \Omega \colon S(\omega) \leq t < T(\omega) \}.
\]
The stochastic intervals \( \gs S, T \of, \of S, T \gs, \gs S, T \gs  \) are defined accordingly.
In particular,  \(\of 0, \infty\of \hspace{0.1cm} = \bR_+ \times \Omega\). 

Let \(K\) and \(F\) be topological spaces and 
let \(b \colon K \times F \times \of 0, \infty\of \hspace{0.05cm}\to \mathbb{R}^\d\) and \(\sigma \colon K \times F \times \of 0, \infty\of \hspace{0.05cm}\to \mathbb{R}^{d \times r}\) be Borel functions such that \((t, \omega) \mapsto b(k, f, t, \omega)\) and \((t, \omega) \mapsto \sigma (k, f, t, \omega)\) are \(\mathscr{P}\)-measurable for every \((k,f) \in K\times F\).\footnote{By \cite[Theorem IV.97]{DM}, a Borel map \(f \colon \of 0, \infty\of \hspace{0.05cm} \to \bR\) is \(\mathscr{P}\)-measurable if and only if \(f (t, \omega)\) depends on \((t, \omega)\) only through the values \((\omega (s))_{s < t}\). This characterization of predictability was used in Section \ref{sec: vis}.} 
We define a correspondence, i.e., a set-valued mapping, \(\Theta \colon K \times \of 0, \infty\of \hspace{0.05cm} \twoheadrightarrow \mathbb{R}^\d \times \bR^{d \times d}\) by
\[
\Theta (k, t, \omega) := \big\{(b (k, f, t, \omega), \sigma \sigma^* (k, f, t, \omega)) \colon f \in F \big\}.
\]
Below, we will impose the following conditions on \(K, F, b\) and \(\sigma\).

\begin{condition} \label{cond: main1}
	\quad
	\begin{enumerate}
		\item[\textup{(i)}]
		\(K\) is first countable and \(F\) is a compact metrizable space.
		\item[\textup{(ii)}] 
		The functions \(b\) and \(\sigma\) are continuous on \(K \times F \times \of 0, \infty\of\).
		\item[\textup{(iii)}] \(\Theta\) is convex-valued. 
		\item[\textup{(iv)}] 
		For every \(T > 0\) and every compact set \(G \subset K\), there exists a constant \(C = C (T, G) > 0\) such that
		\[
		\|b (k, f, t, \omega)\| + \|\sigma (k, f, t, \omega)\| \leq C \Big( 1 + \sup_{s \in [0, t]} \|\omega (s)\| \Big)
		\]
		for all \((k, f, t, \omega) \in G \times F \times \of 0, T\gs\).
		\item[\textup{(v)}]
		For every \(k \in K\) and \(T > 0\), there exists a constant \(C = C (k, T) > 0\) such that 
		\begin{align*}
			\| b (k, f, t, \omega) - b (k, f, t, \alpha) \| &\leq C \ \sup_{s \in [0, t]} \|\omega (s) - \alpha (s)\|, \\
			\|\sigma (k, f, t, \omega) - \sigma (k, f, t, \alpha)\| &\leq C \ \sup_{s \in [0, t]} \|\omega (s) - \alpha (s)\|
		\end{align*}
		for all \(f \in F, t \in [0, T]\) and \(\omega, \alpha \in \Omega \colon \sup_{s \in [0, t]} \|\omega (s)\| \vee \|\alpha (s)\| \leq T\).
	\end{enumerate}
\end{condition}
We draw the readers attention to the fact that the local Lipschitz constant from~(v) is uniform in \(f\) but allowed to depend on~\(k\).

We call an \(\bR^\d\)-valued continuous process \(Y = (Y_t)_{t \geq 0}\) a (continuous) \emph{semimartingale after a time \(t^* \in \mathbb{R}_+\)} if the process \(Y_{\cdot + t^*} = (Y_{t + t^*})_{t \geq 0}\) is a \(\d\)-dimensional semimartingale for its natural right-continuous filtration.
The law of a semimartingale after \(t^*\) is said to be a \emph{semimartingale law after \(t^*\)} and the set of them is denoted by \(\fPs (t^*)\).
For \(P \in \fPs (t^*)\) we denote the semimartingale characteristics of the shifted coordinate process \(X_{\cdot + t^*}\) by \((B^P_{\cdot + t^*}, C^P_{\cdot + t^*})\). 
Moreover, we set 
\begin{align*}
	\fPas (t^*) &:= \big\{ P \in \fPs (t^*) \colon P\text{-a.s. } (B^P_{\cdot + t^*}, C^P_{\cdot + t^*}) \ll \llambda \big\}, \\ 
	\fPas &:= \fPas (0),
\end{align*}
where \(\llambda\) denotes the Lebesgue measure. 
Finally, for \(k \in K\) and $(t,\omega) \in \of 0, \infty\of$, we define 
\begin{align*}
	\cC(k, t,\omega) := \Big\{ P \in \mathfrak{P}_{\text{sem}}^{\text{ac}}(t)\colon &P(X = \omega \text{ on } [0, t]) = 1, 
	(\llambda \otimes P)\text{-a.e. } \\&\qquad (dB^P_{\cdot + t} /d\llambda, dC^P_{\cdot + t}/d\llambda) \in \Theta (k, \cdot + t, X) \Big\}.
\end{align*}

Fix a bounded continuous function \(\psi \colon K \times \of 0, \infty\of \hspace{0.05cm} \to \mathbb{R}\) and define the so-called \emph{extended value function} by
\[
v (k, t, \omega) := \sup_{P \in \cA (k, t, \omega)} E^P \big[ \psi (k, t, X) \big], \quad (k, t, \omega) \in K \times \of 0, \infty\of.
\]
Various properties of value functions without \(k\)-dependence have been studied in~\cite{CN23a}. In this paper, we are explicitly interested in the \(k\)-variable. The following theorem shows that the value function is {\em jointly continuous} under Condition~\ref{cond: main1}. Its proof is given in Section \ref{sec: pf regularity}.
\begin{theorem} \label{thm: main1}
	Assume that Condition \ref{cond: main1} holds. Then, the extended value function
	\(
	(k, t, \omega) \mapsto v (k, t, \omega)
	\)
	is bounded and continuous. 
\end{theorem}

For every \(T > 0\), we define \(\dd = \dd_T \colon \of 0, T \gs \times \of 0, T \gs \to \mathbb{R}_+\) by the formula \eqref{eq: def d}. The following result is a slight variation of \cite[Theorem 4.12]{CN23a}. We sketch the proof in Section \ref{sec: pf Lip cont}.

\begin{condition} \label{cond: glob Lip}
	For every \(k \in K\) and \(T > 0\), there exists a constant \(C = C (k, T) >0\) such that 
	\begin{align*}
		\| b (k, f, t, \omega) - b (k, f, t, \alpha) \| &\leq C \ \sup_{s \in [0, t]} \|\omega (s) - \alpha (s)\|,\\
		\|\sigma (k, f, t, \omega) - \sigma (k, f, t, \alpha)\| &\leq C \ \sup_{s \in [0, t]} \|\omega (s) - \alpha (s)\|
	\end{align*}
	for all \(f \in F, t \in [0, T]\) and \(\omega, \alpha \in \Omega\).
\end{condition}

\begin{theorem} \label{thm: lip cont VF}
	Assume that the Conditions \ref{cond: main1} and \ref{cond: glob Lip} hold. 
	Furthermore, assume that \(\psi\) is independent of time, i.e., a bounded continuous function \(\psi \colon K \times \Omega \to \bR\), and that \(T > 0\) is such that \(\psi (k, \omega)\) depends on \(\omega\) only through \((\omega (s))_{s \leq T}\). Finally, assume that for every \(k \in K\) the function \(\omega \mapsto \psi (k, \omega)\) is Lipschitz continuous, i.e., there exists a constant \(C = C_k > 0\) such that
	\[
	|\psi (k, \omega) - \psi (k, \alpha)| \leq C \sup_{s \in [0, T]} \|\omega (s) - \alpha (s)\|
	\]
	for all \(\omega, \alpha \in \Omega\).
	Then, for every \(k \in K\), there exists a constant \(L = L (k, T) > 0\) such that 
	\[
	|v (k, t, \omega) - v (k, s, \alpha)| \leq L\hspace{0.025cm} \dd_T ( (t, \omega), (s, \alpha) )
	\]
	for all \((t, \omega), (s, \alpha) \in \of 0, T\gs\).
\end{theorem}
Notice that the constant \(L\) from Theorem \ref{thm: lip cont VF} depends on \(k\). This dependence is inherited from the Lipschitz constants of \(b, \sigma\) and \(\psi\). 
Next, we deduce Theorem \ref{thm: stability} from the Theorems~\ref{thm: main1} and \ref{thm: lip cont VF}.


\subsection{Proof of Theorem \ref{thm: stability}} \label{sec: pf stability}
Let \(v^0, v^1, \dots\) be as in Theorem \ref{thm: stability}.
Set \(K := \{1 /n \colon n \in \mathbb{N}\} \cup \{0\}\) and endow it with the Euclidean topology. Then, \(K\) is a compact metrizable space, which shows that Condition \ref{cond: main1} (i) holds. We define
\begin{align*}
	b (1/n, f, t, \omega) &:= b^n (f, t \wedge T, \omega (\cdot \wedge T)), \\
	b (0, f, t, \omega) &:= b^0 (f, t  \wedge T, \omega (\cdot \wedge T)), \\
	\sigma (1/n, f, t, \omega) &:= \sigma^n (f, t \wedge T, \omega(\cdot \wedge T)) , \\
	\sigma (0, f, t, \omega) &:= \sigma^0 (f, t \wedge T, \omega(\cdot \wedge T)), \\
	\psi (1/n, t, \omega) &:= \psi^n (\omega (\cdot \wedge T)), \\
	\psi (0, t, \omega) &:= \psi^0 (\omega (\cdot \wedge T)).
\end{align*} 
It follows from Condition \ref{cond: main2} (ii) and Condition~\ref{cond: main3} (i) that \(b, \sigma\) and \(\psi\) are continuous. Consequently, Condition \ref{cond: main1} (ii) holds. We also note that \(\psi\) is bounded by Condition \ref{cond: main3} (iii). Moreover, Condition~\ref{cond: main1} (iii) is implied by Condition \ref{cond: main2}~(iii); Condition~\ref{cond: main1} (iv) is implied by Condition \ref{cond: main3} (ii); and the Conditions~\ref{cond: main1} (v) and~\ref{cond: glob Lip} follow from Condition~\ref{cond: main2} (v). In summary, the Conditions~\ref{cond: main1} and \ref{cond: glob Lip} are satisfied.
It follows from \cite[Theorem 4.3]{CN23a} and the Theorems~\ref{thm: uni} and \ref{thm: lip cont VF}, that, for \((t, \omega) \in \of 0, T \gs\),
\begin{align*}
	v^n (t, \omega) &= \sup_{P \in \mathcal{C} (1/n, t, \omega)} E^P \big[ \psi (1/n, t, X) \big], \\ v^0 (t, \omega) &= \sup_{P \in \mathcal{C} (0, t, \omega)} E^P \big[ \psi (0, t, X)\big].
\end{align*}
Theorem \ref{thm: main1} yields that for every sequence \((t^n, \omega^n)_{n = 0}^\infty \subset \of 0, T\gs\) such that \((t^n, \omega^n) \to (t^0, \omega^0)\) it holds that 
\[
v^n (t^n, \omega^n) \to v^0 (t^0, \omega^0). 
\]
By \cite[Theorem on pp. 98--99]{remmert}, this implies that \(v^n \to v^0\) uniformly on compact subsets of \(\of 0, T\gs\). The proof is complete. 
\qed

\subsection{Proof of Theorem \ref{thm: main1}} \label{sec: pf regularity}
Our strategy is to use Berge's maximum theorem. A similar idea was used in \cite{CN23a} to establish the continuity of the standard (i.e., \(k\)-independent) value function. To the best of our knowledge, for stochastic frameworks the strategy traces back to the seminal paper~\cite{nicole1987compactification}.

Let us start with another representation of the extended value function that appears to be more convenient. 
For \(\omega, \omega' \in \Omega, t \in \mathbb{R}_+\) and \(k \in K\), we define the concatenation
\[
\omega\ \widetilde{\otimes}_t\ \omega' :=  \omega \1_{[ 0, t)} + (\omega (t) + \omega' (\cdot - t) - \omega' (0)) \1_{[t, \infty)},
\]
and the set
\begin{equation*}
	\begin{split}
		\mathcal{R}(k, t,\omega) := \Big\{ P \in \fPas \colon &P \circ X_0^{-1} = \delta_{\omega (t)},\
		(\llambda \otimes P)\text{-a.e. } \\&(dB^{P} /d\llambda, dC^{P}/d\llambda) \in \Theta (k, \cdot + t, \omega \ \widetilde{\otimes}_t\ X)  \Big\}.
	\end{split}
\end{equation*}
The following representation of the extended value function follows from \cite[Corollary 7.3]{CN23a}.
\begin{lemma} \label{lem: chara EVF}
	For all \((k, t, \omega) \in K \times \of 0, \infty\of\), we have
	\[
	v (k, t, \omega) = \sup_{P \in \cR (k, t, \omega)} E^P \big[ \psi ( k, t, \omega \ \widetilde{\otimes}_t \ X) \big]. 
	\]
\end{lemma}

While the correspondence \(\cC\) depends on \(\fPas (t)\) and the \(t\)-dependent characteristics \((dB^{P}_{\cdot + t} /d\llambda,\) \(dC^{P}_{\cdot + t}/d\llambda)\), the main \((k, t, \omega)\)-dependence in \(\cR\) is shifted to the \(\Theta\)-part, which seems to be easier to handle. The following proposition is the key step in the proof of Theorem \ref{thm: main1}. Its proof is postponed to the end of this section.
\begin{theorem} \label{prop: C continuous}
	Assume that Condition \ref{cond: main1} holds. Then, the correspondence \(\cR\) is continuous (\cite[Definition 17.2]{charalambos2013infinite}) with nonempty and compact values. 
\end{theorem}

For reader's convenience, we recall a version of  {\em Berge's maximum theorem}, see \cite[Theorem~17.31]{charalambos2013infinite}.

\begin{theorem} [Berge's Maximum Theorem] \label{theo: berge}
	Let \(G\) be a first countable topological space, let \(H\) be a metrizable space and let \(\Xi \colon G \to H\) be a continuous correspondence with nonempty and compact values. For every continuous function \(g \colon \on{gr} \Xi \to \mathbb{R}\), the map
	\[
	G \ni x \mapsto \max_{y \in \Xi (x)} g (x, y) \in \mathbb{R}
	\]
	is continuous. 
\end{theorem}

By virtue of Lemma \ref{lem: chara EVF}, the continuity of the extended value function  follows from Theorem \ref{theo: berge} applied with the following input:
\begin{enumerate}
	\item[-] \(G \equiv K \times \of 0, \infty\of\), which is first countable by part (i) of Condition \ref{cond: main1} and the fact that products of first countable spaces remain first countable;
	\item[-] \(H \equiv \mathfrak{P}(\Omega)\), which is well-known to be a Polish space;
	\item[-] \(\Xi \equiv \cR\), which is a continuous correspondence with nonempty and compact values by Theorem~\ref{prop: C continuous};
	\item[-] \(g (k, t, \omega, P) \equiv E^P [ \psi ( k, t,  \omega \ \widetilde{\otimes}_t \ X) ]\), which is continuous thanks to \cite[Theorem~8.10.61]{bogachev} and \cite[Corollary~7.7]{CN23a}.
\end{enumerate}
The proof of Theorem \ref{thm: main1} is complete. \qed
\vspace{0.25cm}

In the remainder of this section, we prove Theorem \ref{prop: C continuous}. 

\begin{proof} [Proof of Theorem \ref{prop: C continuous}]
	The correspondence \(\cR\) has nonempty and compact values by \cite[Lemma~2.10, Theorem~4.4]{CN23a}. Furthermore, upper hemicontinuity follows precisely as in the \(k\)-independent case that is given by \cite[Theorem~4.4]{CN23a}. We omit a detailed proof for brevity. 
	
	It is left to prove that \(\cR\) is lower hemicontinuous. 
	We use \cite[Theorem~17.21]{charalambos2013infinite}, i.e., we prove that for every sequence \((k^n, t^n, \omega^n)_{n = 0}^\infty \subset K \times \of 0, \infty\of\) such that \((k^n, t^n, \omega^n) \to (k^ 0, t^0, \omega^0)\) and every \(P \in \cR (k^0, t^0, \omega^0)\), there exists a subsequence \((k^{N_n}, t^{N_n}, \omega^{N_n})_{n = 1}^\infty\) of \((k^n, t^n, \omega^n)_{n = 1}^\infty\) and a sequence \(P^{N_n} \in \cR (k^{N_n}, t^{N_n}, \omega^{N_n})\) such that \(P^{N_n} \to P\) weakly. We construct the approximation sequence in the spirit of the proof for \cite[Theorem 4.7]{CN23a}.
	However, as the local Lipschitz assumption from part (v) of Condition \ref{cond: main1} is only pointwise in \(k\), we have to use more sophisticated tools to establish weak convergence. Namely, we rely on the concept of weak-strong convergence that we learned from the paper \cite{jacod1981weak}. 
	
	Take a sequence \((k^n, t^n, \omega^n)_{n = 0}^\infty \subset K \times \of 0, \infty\of\) such that \((k^n, t^n, \omega^n) \to (k^0, t^0, \omega^0)\) and a measure \(P \in \cR (k^0, t^0, \omega^0)\).
	
	{\em Step 1: A candidate for an approximation sequence.}
	By \cite[Lemma 8.1]{CN23a}, there exists a predictable map \(\f = \f (P) \colon \of 0, \infty\of \to F\) such that \((\llambda \otimes P)\)-a.e.
	\[
	(dB^P / d \llambda, d C^P / d \llambda) = (b , \sigma \sigma^*) (k^0, \f , \cdot + t^0, \omega^0 \ \widetilde{\otimes}_{t^0} \ X).
	\]
	Take a filtered probability space \((\Omega^*, \cF^*, (\cF^*_t)_{t \geq 0}, P^*)\) that supports an \(r\)-dimensional standard Brownian motion \(W^*\). We set 
	\[
	\Sigma := \Omega \times \Omega^*, \ \ \mathcal{A}:= \cF\otimes \cF^*, \ \ \mathcal{A}_t := \bigcap_{s > t} \big( \cF_s \otimes \cF_s^* \big), \ \ Q := P \otimes P^*.
	\]
	Furthermore, we extend \(X\) and \(W\) to the product space \(\Sigma\) by 
	\[
	X (\omega, \omega^*) := X(\omega), \quad W^*(\omega, \omega^*) := W^* (\omega^*).
	\]
	Then, on the stochastic basis \(\mathbb{B} := (\Sigma, \mathcal{A}, (\mathcal{A}_t)_{t \geq 0}, Q)\), \(W^*\) is a standard Brownian motion and \(X\) is a semimartingale whose characteristics are absolutely continuous with densities 
	\[(b ,\sigma \sigma^*) (k^0, \f (\cdot, X) , \cdot + t^0, \omega^0 \ \widetilde{\otimes}_{t^0} \ X).\] By a standard representation theorem (see \cite[Theorem~7.1', p. 90]{IW} or \cite[Theorem ~9.13]{Kallenberg}), there exists a standard \(r\)-dimensional Brownian motion \(W\) on the stochastic basis \(\mathbb{B}\) such that \(Q\)-a.s.
	\begin{align*}
		d X_t &= b (k^0, \f (t, X), t + t^0, \omega^0 \ \widetilde{\otimes}_{t^0} \ X) dt \\&\hspace{3cm}+ \sigma (k^0, \f (t, X), t + t^0, \omega^0 \ \widetilde{\otimes}_{t^0} \ X) d W_t, 
		\\ 
		X_0 &= \omega^0 (t^0).
	\end{align*}
	Using (iv) and (v) from Condition \ref{cond: main1}, we deduce from \cite[Theorem 14.30]{jacod79} (or \cite[Theorem~4.5]{jacod1981weak}) that, for every \(n \in \mathbb{N}\), there exists a continuous adapted process \(X^n\) on the stochastic basis \(\mathbb{B}\) such that \(Q\)-a.s.
	\begin{align*}
		d X^n_t &= b (k^n, \f (t, X), t + t^n, \omega^n \ \widetilde{\otimes}_{t^n}\ X^n) d t 
		\\&\hspace{3cm} + \sigma(k^n, \f (t, X), t + t^n, \omega^n \ \widetilde{\otimes}_{t^n}\ X^n) d W_t, 
		\\ 
		X^n_0 &= \omega^n (t^n).
	\end{align*}
	Notice that \(X\) and \(X^n\) are driven by the very same Brownian motion \(W\).
	By \cite[Lemma~8.2]{CN23a}, the push-forward measure \(P^n := Q \circ (X^n)^{-1}\) is an element of \(\cR (k^n, t^n, \omega^n)\). Any convergent subsequence of \((P^n)_{n = 1}^\infty\) appears to be a candidate for an approximation sequence of \(P\). 
	
	To motivate the next step, we draw the reader's attention to the fact that the coefficients of the SDE for \(X^n\) are random in the sense that they not only depend on the paths of \(X^n\) but also on the paths of \(X\). In particular, as we have no further information on \(\f\) rather than its predictability, there is no continuity of the coefficients in the \(X\)-variable. This lack of regularity prevents us to use standard arguments based on weak convergence (convergence in distribution) and the usual continuous mapping theorem. Instead, we pass to a stronger type of convergence, the so-called {\em weak-strong convergence}. This type of convergence has a powerful continuous mapping theorem for Carath\'eodory functions, which is tailor made for our needs.
	
	{\em Step 2: Selecting a suitable subsequence.} 
	As explained above, to select a suitable convergent subsequence from \((P^n)_{n = 1}^\infty\), we use the concept of weak-strong convergence that we shortly recall in the following. Set \((U, \mathcal{U}) := (\Sigma \times \Omega, \mathcal{A} \otimes \mathcal{B}(\Omega))\). We say that a sequence \((R^n)_{n = 1}^\infty\) of probability measures on \((U, \mathcal{U})\) converges in the weak-strong sense to a probability measure \(R\) on \((U, \mathcal{U})\) if 
	\[
	E^{R^n} \big[ f \big] \to E^R \big[ f \big]
	\]
	for all bounded Carath\'eodory functions \(f \colon U \to \mathbb{R}\) that are measurable in the first and continuous in the second variable. 
	Define \(R^n := Q \circ (\on{id}, X^n)^{-1}\), which is a sequence of probability measures on \((U, \mathcal{U})\). Similar to the proof of \cite[Lemma 7.4]{CN23a}, part (iv) from Condition \ref{cond: main1} shows that the set \(\{P^n \colon n \in \mathbb{N}\}\) is relatively compact in \(\mathfrak{P}(\Omega)\). The sequence of \(\Sigma\)-marginals of \((R^n)_{n = 1}^\infty\) is given by the constant sequence \((Q)_{n = 1}^\infty\) that is trivially sequentially relatively compact in the space of probability measures on \((\Sigma, \mathcal{A})\) endowed with the topology of set-wise convergence. 
	Thus, by \cite[Theorem 2.5]{CPS}, there exists a subsequence \((R^{N_n})_{n = 1}^\infty \subset (R^n)_{n = 1}^\infty\) that converges in the weak-strong sense to a probability measure \(R\). The sequence \((P^{N_n})_{n = 1}^\infty\) is our candidate for an approximating sequence of the measure~\(P\). 
	
	{\em Step 3: \(P^{N_n} \to P\) weakly.} We use a martingale problem argument to identify the measure \(R\). 
	Define (extended) coefficients \(\ob \colon K \times F \times \of 0, \infty\of \hspace{0.05cm}\to \bR^{r + d}\) and \(\osigma \colon K \times F \times \of 0, \infty\of \hspace{0.05cm}\to \bR^{(r + d) \times r}\) by 
	\[
	\ob := \binom{0}{b}, \quad \osigma := \binom{\on{id}}{\sigma}.
	\] 
	These extensions allow us to take the driving Brownian motion into account.
	Let \(C^2_b(\mathbb{R}^{r + d}; \bR)\) be the space of all bounded twice continuously differentiable functions from \(\bR^{r + d}\) into \(\bR\) with bounded gradient and bounded Hessian matrix. Fix a test function \(g \in C^2_b (\bR^{r + d}; \bR)\) and two deterministic times \(0 \leq S < T < \infty\). Furthermore, take \(G \in \mathcal{A}\) and let \(h \colon \Omega \to \mathbb{R}\) be a bounded continuous function such that \(h (\omega)\) depends on \(\omega\) only via \((\omega (t))_{t \leq S}\).
	For \(n \in \mathbb{N}\) and \((\alpha, \gamma) \in U\), we set 
	\begin{align*}
		g^n_s& (\alpha, \gamma) \\
		&:= \langle \nabla g (W_s (\alpha), \gamma (s)), \ob (k^n, \f (s, X(\alpha)), s + t^n, \omega^n \ \widetilde{\otimes}_{t^n} \ \gamma) \rangle \\& \hspace{1cm}+ \tfrac{1}{2} \on {tr} \big[ \nabla^2 g(W_s (\alpha), \gamma(s))\ \osigma \osigma^* (k^n, \f (s, X(\alpha)), s + t^n, \omega^n \ \widetilde{\otimes}_{t^n} \ \gamma) \big],
	\end{align*}
	and
	\[
	M^n_t (\gamma) (\alpha) := g (W_t (\alpha), \gamma(t)) - \int_0^t  g^n_s (\alpha, \gamma) ds, \quad t \in \bR_+.
	\]
	By It\^o's formula, 
	\begin{align} \label{eq: ito formula}
		d M^n_t (X^n) = \langle \nabla g (W_t, X^n_t), \osigma (k^n, \f (t, X), t + t^n, \omega^n \ \widetilde{\otimes}_{t^n} \ X^n) d W_t \rangle.
	\end{align}
	Thus, each \(M^n (X^n)\) is a local martingale on the stochastic basis \(\mathbb{B}\). 
	Using (iv) from Condition~\ref{cond: main1} and a standard Gronwall argument as in the solution to \cite[Problem~3.3.15]{KaraShre}, we obtain that 
	\begin{align} \label{eq: poly moments}
		\sup_{n \in \mathbb{Z}_+} E^Q \Big[ \sup_{s \in [0, t]} \|X^n_s\|^p \Big] < \infty, \quad t > 0, \ p \geq 2.
	\end{align}
	It follows from \eqref{eq: ito formula} and Condition \ref{cond: main1} (iv) that \(Q\)-a.s. for all \(t \in \mathbb{R}_+\)
	\begin{align*}
		[ M (X^n)&, M (X^n) ]_t \\&= \int_0^t  \| \osigma^* (k^n, \f (s, X), s + t^n, \omega^n \ \widetilde{\otimes}_{t^n} \ X^n) \nabla g (W_s, X^n_s)\|^2 ds 
		\\&\leq C \Big( 1 + \sup_{s \in [0, t]} \|X^n_s\|^2 \Big),
	\end{align*}
	where \([\ \cdot \ , \ \cdot\ ]\) denotes the quadratic variation process. Hence, with \eqref{eq: poly moments}, we get
	\[E^Q\big[ [M (X^n), M(X^n) ]_t \big] < \infty,\quad t \in \mathbb{R}_+,\] 
	which implies that \(M (X^n)\) is a true martingale on \(\mathbb{B}\). 
	Consequently, 
	\begin{align} \label{eq: MP approx}
		E^{Q} \big[ (M^n_T (X^n) - M^n_S (X^n)) \1_G h (X^n) \big] = 0.
	\end{align}
	Observe that 
	\begin{align*}
		\big| &M^n_T (X^n) - M^0_T (X^n) \big| 
		\\&\ \ \leq C \int_0^T \sup_{f \in F} \| b (k^n, f, s + t^n, \omega^n \ \widetilde{\otimes}_{t^n} \ X^n) - b (k^0, f, s + t^0, \omega^0 \ \widetilde{\otimes}_{t^0} \ X^n) \|
		\\&\hspace{2cm}+ \sum_{i, j = 1}^{r + d} \sum_{l = 1}^r \sup_{f \in F}\big| \osigma^{(il)} \osigma^{(jl)} (k^n, f, s + t^n, \omega^n \ \widetilde{\otimes}_{t^n} \ X^n) 
		\\&\hspace{5cm}
		- \osigma^{(il)} \osigma^{(jl)} (k^0, f, s + t^0, \omega^0 \ \widetilde{\otimes}_{t^0} \ X^n) \big| ds
		\\&\ \ =: V_n,
	\end{align*}
	where \(C > 0\) is a constant that only depends on the test function \(g\). 
	Using part (iv) from Condition~\ref{cond: main1} and \eqref{eq: poly moments}, we obtain that 
	\begin{align} \label{eq: N moment bound}
		\sup_{n \in \mathbb{N}} E^Q \big[ V^2_n \big] < \infty.
	\end{align}
	By Condition \ref{cond: main1} (i)+(ii), \cite[Corollary 7.7]{CN23a} and Theorem \ref{theo: berge}, for every \(s \in \bR_+\), the map 
	\begin{align*}
		(k, t, \omega, \alpha) \mapsto  \sup_{f \in F} \| b (k&, f, s + t, \omega \ \widetilde{\otimes}_t\ \alpha ) - b (k^0, f, s+ t^0, \omega^0 \ \widetilde{\otimes}_{t^0} \ \alpha)\| 
		\\&+ \sum_{i, j = 1}^{r + d} \sum_{l = 1}^r \sup_{f \in F} \big| \osigma^{(il)}\osigma^{(jl)} (k, f, s + t, \omega \ \widetilde{\otimes}_t\ \alpha) 
		\\&\hspace{2.5cm}- \osigma^{(il)}\osigma^{(jl)} (k^0, f, s + t^0, \omega^0 \ \widetilde{\otimes}_{t^0} \ \alpha) \big|
	\end{align*}
	is continuous. The weak-strong convergence of \((R^{N_n})_{n = 1}^\infty\) clearly implies weak convergence of the sequence \((P^{N_n})_{n = 1}^\infty\).
	Thus, by Skorokhod's coupling theorem, we may realize the sequence \((X^{N_n})_{n = 1}^\infty\) on a common probability space such that it converges almost surely. Using this realization, we get that a.s. \(V_{N_n} \to 0\) and, as the set \(\{V_n \colon n \in \mathbb{N}\}\) is uniformly integrable by \eqref{eq: N moment bound}, Vitali's theorem implies that \(E^Q[V_{N_n}] \to 0\).
	We conclude that
	\begin{align} \label{eq: conv 1}
		E^Q \big[ \big| M^{N_n}_T (X^{N_n}) - M^0_T (X^{N_n}) \big| \big] \to 0. 
	\end{align}
	Of course, the same statement holds true when the time \(T\) is replaced by \(S\).
	For \((\alpha, \gamma) \in U\), define
	\begin{align*}
		D^0 &(\alpha, \gamma) \\&:= \Big(g (W_T (\alpha), \gamma (T)) - g (W_S (\alpha), \gamma (S)) - \int_S^T g^0_s (\alpha, \gamma) ds\Big) \1_G (\alpha) h (\gamma).
	\end{align*}
	For every \(\alpha \in \Sigma\), the map \(\gamma \mapsto D (\alpha, \gamma)\) is continuous by Condition~\ref{cond: main1}~(ii), \cite[Corollary 7.7]{CN23a} and the continuity of \(h\) and \(g\). Hence, using again \eqref{eq: poly moments} and Condition~\ref{cond: main1}~(iv) for uniform integrability, we deduce from the Jacod--M\'emin continuous mapping theorem for weak-strong convergence (\cite[Theorem~2.9]{CPS}) that 
	\begin{align*}
		E^R \big[ D^0 \big] &= \lim_{n \to \infty} E^{R^{N_n}} \big[ D^0 \big] 
		\\&= \lim_{n \to \infty} E^Q \big[ (M^0_T (X^{N_n}) - M^0_S (X^{N_n})) \1_G  h (X^{N_n}) \big].
	\end{align*}
	Finally, using \eqref{eq: MP approx} and \eqref{eq: conv 1}, we conclude that
	\begin{align*}
		\big| E^Q \big[ (&M^0_T (X^{N_n}) - M^0_S (X^{N_n})) \1_G h (X^{N_n}) \big] \big| 
		\\& = \big|E^Q \big[ (M^0_T (X^{N_n}) - M^{N_n}_T (X^{N_n}) 
		\\&\hspace{3cm}+ M^{N_n}_S (X^{N_n}) - M^0_S (X^{N_n})) \1_G h (X^{N_n}) \big]\big|
		\\&\leq C \big( E^Q \big[ \big| M^0_T (X^{N_n}) - M^{N_n}_T (X^{N_n}) \big| \big] \\&\hspace{3cm}+ E^Q \big[ \big| M^0_S (X^{N_n}) - M^{N_n}_S (X^{N_n}) \big| \big]  \big)
		\\&\xrightarrow{\quad} 0.
	\end{align*}
	This proves that \(E^R[ D^0 ] = 0\). By a monotone class argument, we conclude from \cite[Theorem~2.10]{jacod1981weak} that \(R\) is a good solution measure (see \cite[Definition~1.7]{jacod1981weak}) to the~SDE
	\begin{align*}
		d Y_t &= b (k^0, \f (t, X), t + t^0, \omega^0 \ \widetilde{\otimes}_{t^0} \ Y) dt 
		\\&\hspace{3cm}+ \sigma (k^0, \f (t, X), t + t^0, \omega^0 \ \widetilde{\otimes}_{t^0} \ Y) d W_t, 
		\\
		Y_0 &= \omega^0 (t^0).
	\end{align*}
	Thanks to the local Lipschitz condition given by Condition \ref{cond: main1} (v), we deduce from \cite[Theorems~2.22,~2.25(b),~4.5]{jacod1981weak} that 
	\[
	R (d \sigma, d \omega) = \delta_{X (\sigma)} (d \omega) Q (d \sigma).
	\]
	Consequently, we get
	\[
	R (\Sigma \times d \omega) = Q (X \in d \omega) = (P \otimes P^*) (d \omega \times \Omega^*) = P (d \omega), 
	\]
	which implies that weakly
	\[
	P^{N_n} (d \omega) = R^{N_n} (\Sigma \times d \omega) \to R (\Sigma \times d \omega) = P (d \omega).
	\]
	This shows lower hemicontinuity of \(\cR\). The proof of Theorem~\ref{prop: C continuous} is complete.
\end{proof}

\subsection{Proof of Theorem \ref{thm: lip cont VF}} \label{sec: pf Lip cont}
Let \(T > 0\) be as in the hypothesis of the theorem, fix \(k \in K\) and take \((t^0, \omega^0), (s^0, \alpha^0) \in \of 0, T\gs\). W.l.o.g., assume that \(s^0 \leq t^0\). In the following \(L = L (k, T) > 0\) denotes a generic constant (only depending on \(k\) and~\(T\)) that might change from line to line.
It follows from the inequalities (8.8) and (8.9) in \cite{CN23a} that 
\begin{align*}
	|v (k, s^0, \omega^0) - v (k, s^0, \alpha^0)| &\leq L \hspace{0.05cm} \sup_{s \in [0, s^0]} \| \omega^0 (s) - \alpha^0 (s) \|, \\
	| v (k, s^0, \omega^0) - v (k, t^0, \omega^0) | &\leq \sup_{P \in \cA (k, s^0, \omega^0)} E^P \Big[ \sup_{s \in [s^0, t^0]} \| X_s - \omega^0 (s^0) \| \Big] 
	\\&\hspace{3cm}+ \sup_{s \in [s^0, t^0]} \| \omega^0 (s) - \omega^0 (s^0) \|.
\end{align*}
We emphasis that these inequalities were obtained without the boundedness assumptions from \cite[Theorem 4.12]{CN23a}.
Using the linear growth assumptions from Condition \ref{cond: main1}, it follows as in the solution to \cite[Problem 3.3.15]{KaraShre} that
\[
\sup_{P \in \cA (k, s^0, \omega^0)} E^P \Big[ \sup_{s \in [s^0, t^0]} \| X_s - \omega^0 (s^0) \|^2 \Big] \leq L \Big( 1 + \sup_{s \in [0, s^0]} \|\omega^0 (s)\|^2 \Big) |t^0 - s^0|.
\]
By Jensen's inequality, we conclude that 
\begin{align*}
	|v (k, t^0, \omega^0) - v (k, s^0, \alpha^0) | &\leq |v (k, t^0, \omega^0) - v (k, s^0, \omega^0)| 
	\\&\hspace{2cm}+ | v (k, s^0, \omega^0) - v (k, s^0, \alpha^0)| 
	\\&\leq L \Big[ \Big( 1 + \sup_{s \in [0, s^0]} \|\omega^0 (s)\| \Big) | t^0 - s^0|^{1/2}  \\&\hspace{2cm}+ \sup_{s \in [0, t^0]} \| \omega^0 (s) - \alpha^0 (s \wedge s^0)\| \Big]
	\\&\leq L \dd_T ( (t^0, \omega^0), (s^0, \alpha^0)).
\end{align*}
This completes the proof. \qed

\appendix 

\section{The set of test functions} \label{app: test}
The following definitions are adapted from \cite{cosso, cosso2}, see \cite[Appendix~A]{cosso2}.
Let \(T >0\) and \(t_0 \in [0,T)\).
Further, let \(D([0, T]; \bR^\d)\) be the space of \cadlag functions from \([0, T]\) into \(\mathbb{R}^\d\).
We define \(\Lambda(t_0) := [t_0, T] \times D([0, T]; \bR^\d)\) and, on \([0, T] \times D([0, T]; \bR^\d)\), we define the pseudometric \(d^*\) as 
\begin{align*}
	d^* ( (t,\omega), (s, \omega')) := | t - s | + \sup_{r \in [0, T]} \|\omega(r \wedge t) - \omega'( r \wedge s) \|.
\end{align*}
We denote the restriction of \(d^*\) to \(\Lambda(t_0)\) again by \(d^*\).
For a map \(F \colon \Lambda(t_0) \to \bR\) we say that \(F\) admits a \emph{horizontal derivative} at \((t, \omega) \in \Lambda(t_0)\) with \(t < T\) if
\[
\p F (t, \omega) := \lim_{h \searrow 0} \frac{ F(t + h, \omega(\cdot \wedge t)) - F(t, \omega(\cdot \wedge t))}{h}
\]
exists.
At \(t = T\), the horizontal derivative is defined as 
\[
\p F(T, \omega) := \lim_{h \nearrow T} \p F(h,\omega).
\]
Further, we say that 
\(F\) admits a \emph{vertical derivative} at \((t, \omega) \in \Lambda(t_0)\) if 
\[
\partial_i F (t, \omega) := \lim_{h \to 0} \frac{F (t, \omega + h e_i\1_{[t,T]}) - F (t, \omega)}{h}, \quad i = 1, 2, \dots, \d,
\]
exist, where \(e_1, \dots, e_\d\) are the standard unit vectors in \(\bR^\d\).
Accordingly, the second vertical derivatives \(\partial^2_{ij} F(t,\omega), i, j = 1, \dots, \d,\) at \((t, \omega) \in \Lambda(t_0)\) are defined as
\[
\partial^2_{ij} F(t,\omega) := \partial_i (\partial_j F)(t, \omega).
\]
We write \(\nabla F := (\partial_1 F, \dots, \partial_\d F)\) for the {\em vertical gradient} and \[\nabla^2 F := (\partial_{ij}^2 F)_{i, j = 1, \dots, \d}\] for the {\em vertical Hessian matrix}.

We denote by \(C^{1,2}(\Lambda(t_0); \bR)\) the set of functions \(F \colon \Lambda(t_0) \to \bR\), continuous with respect to~\(d^*\), 
such that
\[
\p F, \nabla F, \nabla^2 F
\]
exist everywhere on \(\Lambda(t_0)\) and are continuous with respect to \(d^*\).

The set 
\(C^{1,2}( [ t_0, T ] \times C([0, T]; \bR^d) ; \bR)\) consists of functions 
\(F \colon [ t_0, T ] \times C([0, T]; \bR^d)  \to \bR \)  
such that there exists \(\hat{F} \in C^{1,2}(\Lambda(t_0); \bR)\) with
\[
F(t, \omega) = \hat{F}(t,\omega), \quad (t, \omega) \in [ t_0, T ] \times C([0, T]; \bR^d) .
\]
In this case, we define, for \((t, \omega) \in [ t_0, T ] \times C([0, T]; \bR^d) \),
\[
\p F(t, \omega) := \p \hat{F}(t, \omega), \quad
\nabla F(t,\omega) := \nabla \hat{F}(t,\omega), \quad
\nabla^2 F(t,\omega) := \nabla^2 \hat{F}(t,\omega).
\]
By \cite[Lemma 2.1]{cosso}, \( \cD_t F, \nabla F, \nabla^2 F \) are well-defined for \(F \in C^{1,2}([ t_0, T ] \times C([0, T]; \bR^d) ; \bR)\).

Finally, the set \(C_{pol}^{1,2}( [ t_0, T ] \times C([0, T]; \bR^d) ; \bR) \) consists of all functions \(F\) from \(C^{1,2}( [ t_0, T ] \times C([0, T]; \bR^d) ; \bR)\) such that there exist constants \(C, q \in \bR_+\) with
\[
| \p F(t, \omega) | + \| \nabla F(t,\omega) \| + \on{tr} \big[ \nabla^2 F(t,\omega) \big] \leq C \Big( 1 + \sup_{r \in [t_0,T]} \| \omega( r \wedge t ) \|^q\Big)
\]
for all \((t, \omega) \in [ t_0, T ] \times C([0, T]; \bR^d) \).


\begin{thebibliography}{1}
	\bibitem{charalambos2013infinite}
	C.~D.~Aliprantis and K.~B.~Border. 
	\newblock {\em Infinite Dimensional Analysis: A Hitchhiker's Guide}.
	\newblock Springer Berlin Heidelberg, 3rd ed., 2006.
	
	\bibitem{barles}
	G.~Barles.
	\newblock An Introduction to the Theory of Viscosity Solutions for First-Order Hamilton–Jacobi Equations and Applications.
	\newblock In P.~Loreti, N.~A.~Tchou, editors, {\em Hamilton-Jacobi Equations: Approximations, Numerical Analysis and Applications}, pages 49--109, Springer Berlin Heidelberg, 2013.
	
	
	\bibitem{bogachev}
	V.~I.~Bogachev.
	\newblock {\em Measure Theory}.
	\newblock Springer Berlin Heidelberg, 2007.
	
	
	\bibitem{cosso}
	A.~Cosso and F.~Russo.
	\newblock Crandall--Lions viscosity solutions for path-dependent PDEs: The case of heat equation.
	\newblock {\em Bernoulli}, 28(1):481--503, 2022.
	
	\bibitem{cosso2}
	A.~Cosso, F.~Gozzi, M.~Rosestolato and F.~Russo.
	\newblock Path-dependent Hamilton-Jacobi-Bellman equation:
	Uniqueness of Crandall-Lions viscosity solutions.
	\newblock arXiv:2107.05959v2, 2022.
	
	\bibitem{CL1}
	M.~G.~Crandall and P.~L.~Lions.
	\newblock Condition d'unicit{\'e} pour les solutions g{\'e}n{\'e}ralis{\'e}es des {\'e}quations de {Hamilton}-{Jacobi} du premier ordre.
	\newblock {\em Comptes Rendus de l'Académie des Sciences, Série I - Mathematics}, 292(3):183--186, 1981. 
	
	\bibitem{CL2}
	M.~G.~Crandall and P.~L.~Lions.
	\newblock Viscosity solutions of {Hamilton}-{Jacobi} equations.
	\newblock {\em Transactions of the American Mathematical Society}, 277:1--42, 1983.
	
	\bibitem{CN23a}
	D.~Criens and L.~Niemann.
	\newblock Nonlinear continuous semimartingales.
	\newblock arXiv:2204.07823v3, 2023.
	
	\bibitem{CPS}
	D.~Criens, P.~Pfaffelhuber and T.~Schmidt.
	\newblock The martingale problem method revisited.
	\newblock {\em Electronic Journal of Probability}, 28(19):1--46, 2023.
	
	\bibitem{DM}
	C.~Dellacherie and P.-A.~Meyer.
	\newblock {\em Probability and Potential.}
	\newblock Hermann, Paris, North-Holland Publishing Company - Amsterdam - New York - Oxford, 1978.
	
	\bibitem{nicole1987compactification}
	N.~El Karoui, D. Nguyen and M. Jeanblanc-Picqu{\'e}.
	\newblock Compactification methods in the control of degenerate diffusions: existence of an optimal control.
	\newblock {\em Stochastics}, 20(3):169--219, 1987.
	
	\bibitem{FS}
	W.~H.~Fleming and H.~M.~Soner.
	\newblock {\em Controlled Markov Processes and Viscosity Solutions}.
	\newblock Springer Science+Business Media, 2nd edition, 2006.
	
	\bibitem{IW}
	N.~Ikeda and S.~Watanabe.
	\newblock {\em Stochastic Differential Equatrions and Diffusion Processes}.
	\newblock North-Holland Publishing Company Amsterdam Oxford New York, 2nd edition, 1989.
	
	\bibitem{jacod79}
	J.~Jacod.
	\newblock {\em Calcul stochastique et probl\`emes de martingales}.
	\newblock Springer Berlin Heidelberg New York, 1979.
	
	\bibitem{jacod1981weak}
	J.~Jacod and J.~M\'emin.
	\newblock Weak and strong solutions of stochastic differential equations:
	Existence and stability.
	\newblock In D.~Williams, editor, {\em Stochastic Integrals}, volume 851 of
	{\em Lecture Notes in Mathematics}, pages 169--212. Springer Berlin
	Heidelberg, 1981.
	
	\bibitem{Kallenberg}
	O.~Kallenberg.
	\newblock {\em Foundations of Modern Probability}.
	\newblock Springer New York, 3rd ed., 2021.
	
	\bibitem{KaraShre}
	I.~Karatzas and S.~E.~Shreve.
	\newblock {\em Brownian Motion and Stochastic Calculus}.
	\newblock Springer New York, 2nd edition, 1991.
	
	\bibitem{nutzRGE}
	M.~Nutz.
	\newblock Random \(G\)-Expectation.
	\newblock {\em Annals of Applied Probability}, 23(5):1755--1777, 2013.
	
	\bibitem{NVH}
	M.~Nutz and R.~van Handel.
	\newblock Constructing sublinear expectations on path space. 
	\newblock {\em Stochastic Processes and their Applications}, 123(8):3100--3121, 2013.
	
	\bibitem{peng2007g}
	S.~G.~Peng.
	\newblock G-expectation, G-Brownian motion and related stochastic calculus of It{\^o} type.
	\newblock In F.~E.~Benth et. al., editors, {\em Stochastic Analysis and Applications: The Abel Symposium 2005}, pages 541--567, Springer Berlin Heidelberg, 2007.
	
	\bibitem{peng2008multi}
	S.~G.~Peng.
	\newblock  Multi-dimensional G-Brownian motion and related stochastic calculus under G-expectation.
	\newblock {\em Stochastic Processes and their Applications}, 118:2223--2253, 2008.
	
	\bibitem{peng2010}
	S.~G.~Peng.
	\newblock Nonlinear expectations and stochastic calculus under uncertainty.
	\newblock {\em arXiv:1002.4546}, 2010.
	
	\bibitem{remmert}
	R.~ Remmert.
	\newblock {\em Theory of Complex Functions}.
	\newblock Springer Science+Business Media New York, 1991.
	
	\bibitem{zhou}
	J.~Zhou.
	\newblock Viscosity Solutions to Second Order Path-Dependent Hamilton-Jacobi-Bellman Equations and Applications.
	\newblock arXiv:2005.05309v3, 2022.
\end{thebibliography}
\end{document}